\documentclass[11pt]{article}

\usepackage[margin=1in]{geometry}
\usepackage[T1]{fontenc}
\usepackage{lmodern}
\usepackage{amsmath,amssymb,amsthm}
\usepackage{siunitx}
\usepackage{booktabs}
\usepackage{caption}
\usepackage{flafter}
\usepackage[section]{placeins}
\usepackage{graphicx}
\usepackage{xcolor}
\usepackage{algorithm}
\usepackage{algpseudocode}
\usepackage[numbers,sort&compress]{natbib}
\usepackage{hyperref}
\hypersetup{hidelinks}
\usepackage[capitalise]{cleveref}

\graphicspath{{figures/}}
\newtheorem{lemma}{Lemma}

\newcommand{\rgcx}{RGCX}
\newcommand{\strgcx}{single-temperature RGCX}
\newcommand{\Strgcx}{Single-temperature RGCX}
\newcommand{\ttsfixed}{fixed-length TTS99}
\newcommand{\ttsrenewal}{renewal TTS99}
\newcommand{\AUTHORS}{%
    Hao Geng\textsuperscript{1}, 
    Nan Su\textsuperscript{2}, 
    Gabriele Cavallaro\textsuperscript{1,2}%
}
\newcommand{\AFFILIATIONS}{%
  {\small\textsuperscript{1}Faculty of Electrical and Computer Engineering, University of Iceland, Reykjavik, Iceland} \\[4pt]
  {\small\textsuperscript{2}J{\"u}lich Supercomputing Centre, Forschungszentrum J{\"u}lich, J{\"u}lich, Germany} \\
}

\newcommand{\Rrep}{96}
\newcommand{\kint}{8}
\newcommand{\Tpaper}{0.23}
\newcommand{\startGeightyone}{150}
\newcommand{\startGseventytwo}{100}
\newcommand{\startOther}{50}
\newcommand{\desktopCPU}{Intel Xeon W-1370P}
\newcommand{\desktopClock}{3.6}
\newcommand{\GeightyoneLstar}{680}
\newcommand{\GeightyoneTTS}{12.14}
\newcommand{\GeightyoneTTSlo}{11.26}
\newcommand{\GeightyoneTTShi}{13.08}
\newcommand{\GeightyonePodd}{0.838}
\newcommand{\GseventysevenTTS}{1.93}
\newcommand{\GseventytwoTTS}{3.12}
\newcommand{\GsixtytwoTTS}{2.24}
\newcommand{\GeightyoneMedianHit}{2.54}
\newcommand{\GeightyoneQninety}{6.09}
\newcommand{\VoneDeviation}{1.1}
\newcommand{\exchangeShareLo}{71}
\newcommand{\exchangeShareHi}{74}
\newcommand{\certifiedCount}{9}
\newcommand{\boundGapLo}{1}
\newcommand{\boundGapHi}{4}
\newcommand{\boundGapGfifty}{120}
\newcommand{\GseventysevenLstar}{376}
\newcommand{\GseventysevenPodd}{0.984}
\newcommand{\GseventysevenTTSlo}{1.74}
\newcommand{\GseventysevenTTShi}{2.14}
\newcommand{\GseventysevenMedianHit}{0.90}
\newcommand{\GseventytwoLstar}{400}
\newcommand{\GseventytwoPodd}{0.858}
\newcommand{\GseventytwoTTSlo}{2.88}
\newcommand{\GseventytwoTTShi}{3.36}
\newcommand{\GseventytwoMedianHit}{0.72}
\newcommand{\GsixtytwoLstar}{320}
\newcommand{\GsixtytwoPodd}{0.810}
\newcommand{\GsixtytwoTTSlo}{2.08}
\newcommand{\GsixtytwoTTShi}{2.42}
\newcommand{\GsixtytwoMedianHit}{0.41}
\newcommand{\GeightyoneTTSfresh}{12.28}
\newcommand{\GeightyoneTTSfreshlo}{10.54}
\newcommand{\GeightyoneTTSfreshhi}{14.26}
\newcommand{\EtwoRuns}{512}
\newcommand{\EtwoSmallLo}{0.007}
\newcommand{\EtwoLargeHi}{0.83}
\newcommand{\cosmGeightyone}{1326.1}
\newcommand{\cosmGseventyseven}{34.7}
\newcommand{\cosmGseventytwo}{12.9}
\newcommand{\zickGeightyone}{87600}
\newcommand{\zickGseventyseven}{5320}
\newcommand{\zickGseventytwo}{1650}
\newcommand{\gfGeightyone}{14651}
\newcommand{\gfGseventyseven}{287}
\newcommand{\gfGseventytwo}{147}
\newcommand{\gfGsixtytwo}{72}
\newcommand{\abRuns}{256}
\newcommand{\fixedRS}{1.69}
\newcommand{\fixedRSlo}{1.30}
\newcommand{\fixedRShi}{2.29}
\newcommand{\fixedWall}{1.74}
\newcommand{\fixedWalllo}{1.36}
\newcommand{\fixedWallhi}{2.42}
\newcommand{\ascRS}{5.9}
\newcommand{\ascWall}{6.7}
\newcommand{\qkeepRS}{9.3}
\newcommand{\qkeepWall}{7.9}
\newcommand{\qkeepHits}{103}
\newcommand{\parentRS}{3.78}
\newcommand{\parentWall}{2.65}
\newcommand{\parentGseventytwoRS}{1.02}
\newcommand{\parentGseventytwoWall}{0.67}
\newcommand{\GeightyoneOpt}{14060}
\newcommand{\priorBestLo}{13992}
\newcommand{\priorBestHi}{14058}
\newcommand{\cosmRatioGseventyseven}{18}
\newcommand{\cosmRatioGseventytwo}{4.1}

\newcommand{\mcuMedian}{31}
\newcommand{\mcuPninety}{70}
\newcommand{\mcuTTS}{124}
\newcommand{\mcuTTSlo}{106}
\newcommand{\mcuTTShi}{144}
\newcommand{\mcuClock}{64}
\newcommand{\mcuArena}{192}
\newcommand{\mcuLstar}{440}
\newcommand{\mcuBoardLo}{17}
\newcommand{\mcuBoardHi}{93}
\newcommand{\mcuModelLo}{0.6}
\newcommand{\mcuModelHi}{0.9}

\newcommand{\mcuHits}{496}

\newcommand{\clusterCores}{64}
\newcommand{\clusterClock}{2.25}
\newcommand{\mainRunLimit}{4000}
\newcommand{\abRunLimit}{10000}
\newcommand{\abBudgetRatio}{2.5}
\newcommand{\cosmRatioLowerGeightyone}{100}

\title{Single-Temperature Ranked Greedy Cluster Exchange for Max-Cut on Two-Dimensional Tori}
\author{\AUTHORS \\[6pt] \AFFILIATIONS}
\date{}

\begin{document}
\maketitle

\begin{abstract}

Ground-state search in two-dimensional Ising spin glasses on large toroidal lattices remains challenging for local-update heuristics because rugged energy barriers impede relaxation across metastable states.
Equivalent to zero-field toroidal Max-Cut problems, these spin-glass instances also serve as standard benchmarks for optimization heuristics.
Many Monte Carlo approaches combine single-spin flips or cluster moves with temperature ladders or annealing schedules, yet even with tuned parameters, rarely achieve rapid, reliable ground-state convergence at scale.
Here, we address the need for faster ground-state search through cooperation between replicas. We introduce ranked greedy cluster exchange (\rgcx) and combine it with fixed-temperature Metropolis sweeps to form \strgcx.
Replicas are ranked by cut to determine the pair-processing order. Each pair deterministically exchanges every disagreement cluster that raises the better replica’s cut, retaining the complement and conserving the pair’s summed cut. 
On the 20,000-spin Gset instance G81, the estimated runtime to reach the certified optimum with 99\% success probability (TTS99) was \num[round-mode=places,round-precision=1]{\GeightyoneTTS} s on an 8-core desktop computer.
With the same population size, exchange interval and temperature, the method reached the best known cuts on all 17 toroidal Gset graphs, including \certifiedCount{} certified optima.
These findings advance ground-state search on large two-dimensional tori by establishing single-temperature RGCX as an efficient population-based method that combines ranked greedy cluster exchange with fixed-temperature Metropolis dynamics.
\end{abstract}

\section{Introduction and related work}
\label{sec:intro}

The Max-Cut problem seeks a partition of a graph's vertex set that maximizes the total weight of the edges crossing the cut.
On a two-dimensional nearest-neighbor lattice with edge weights $\pm1$, maximizing the cut is equivalent to finding a ground state of the zero-field Edwards--Anderson Ising model~\cite{edwards1975theory,barahona1988application}.
Such lattices are a standard test for both optimization heuristics and Ising machines~\cite{mohseni2022ising}.
The Gset collection~\cite{helmberg2000spectral} contains 17 toroidal lattices of degree four, 14 of them with weights $\pm1$ and three with unit weights, from $N=800$ to $N=\num{20000}$ vertices.

For the largest of them, G81, earlier studies reported cuts below the value now known to be optimal.
These range from \num{\priorBestLo} for discrete simulated bifurcation~\cite{goto2021high} to \num{\priorBestHi} for the mixed-integer solver Gurobi in a recent comparison~\cite{kuroki2026annealed}, with other methods in between~\cite{benlic2013breakout,shylo2015teams,goudet2024large}.
The value \num{\GeightyoneOpt} was first reached in 2025~\cite{zick2025performance}, and in 2026 upper bounds proved it optimal~\cite{glass2026maximum}.
Published time-to-solution (TTS) values to reach this optimum range from \SI{\cosmGeightyone}{\second}, computed for a 32-core server~\cite{zick2026cosm}, to about a day~\cite{zick2025performance}.

Monte Carlo approaches to such problems typically combine single-spin flips or cluster moves with a ladder of temperatures or an annealing schedule~\cite{zhu2015efficient,aadit2026programmable}.
The ladder and the schedule introduce additional parameters, including the number and spacing of the temperatures and the length of the schedule, which are tuned to the problem at hand~\cite{katzgraber2006feedback,wang2015comparing}.

Recombining two configurations by selecting clusters from either parent has a long history.
Kawashima and Suzuki compared replicas and flipped each cluster of their disagreement set that lowered the energy, at zero temperature~\cite{kawashima1992replica}.
Iterative partial transcription improves one state by copying favorable clusters from others, inside a thermal search~\cite{mobius1999combinatorial}.
Houdayer's move picks a random site of the disagreement set of two replicas at the same temperature and exchanges the cluster that contains it~\cite{houdayer2001cluster}, and isoenergetic cluster moves extend the use of such updates to spin glasses beyond two dimensions~\cite{zhu2015efficient}.
Houdayer and Martin searched for ground states with a population whose offspring had to improve on their parents~\cite{houdayer2001hierarchical}.
In evolutionary computation, partition crossover returns the component-wise best child of two parents~\cite{tinos2015partition,whitley2023partition}, while the Mixing Genetic Algorithm additionally retains the complementary child, preserving the population's mean objective value during recombination, and organizes pairings so that better solutions tend to meet~\cite{varadarajan2019massively}.
Houdayer moves have also been adopted as crossover in a quantum-assisted genetic algorithm~\cite{king2019quantum}.
Generalized Houdayer moves allow combinations of disagreement clusters~\cite{vandenbroucque2022houdayer}.
The same study discusses single-temperature populations with local updates, as well as all-pairs crossover in a genetic algorithm.
For samples from quantum annealers, Houdayer moves were used to improve diversity~\cite{ochoa2019feeding}, and greedy choices of disagreement components reduce a set of samples pair by pair to one improved sample~\cite{dorband2018method,ayanzadeh2021multiqubit}.
Ochoa, Schuetz and Katzgraber applied Houdayer-based post-processing to samples from Rydberg atom arrays and proposed retaining only the better member of each exchanged pair as a possible optimization strategy, which they did not test~\cite{ochoa2026efficient}.

In this paper, we propose single-temperature ranked greedy cluster exchange (\strgcx) to couple a population of fixed-temperature Metropolis replicas through our ranked greedy cluster exchange (\rgcx) operator.
For an aligned pair, \rgcx's greedy rule produces the best cluster-wise recombination, as in partition crossover~\cite{tinos2015partition,whitley2023partition}.
It retains the complementary worst child as well, conserving the pair's summed cut~\cite{varadarajan2019massively,whitley2023partition}.
Each exchange call ranks the replicas by cut and processes all pairs in lexicographic rank order, keeping the ranks fixed throughout the call.
This pairing schedule differs from the hypercube pairings with periodic population reshuffling used in the Mixing Genetic Algorithm~\cite{varadarajan2019massively}.
\rgcx\ calls alternate with Metropolis sweeps at a fixed temperature, without a temperature ladder or an annealing schedule.

Our contributions are the following.
\begin{itemize}
  \item We introduce \strgcx\ and establish the properties of its pair operation (\cref{sec:method}).
  \item On an 8-core desktop computer, the estimated \ttsfixed\ of \strgcx\ for G81 is \SI[round-mode=places,round-precision=1]{\GeightyoneTTS}{\second}, about two orders of magnitude below the shortest published TTS99.
  Using the same population size, exchange interval and temperature, the method reaches the best known cuts on all 17 toroidal Gset graphs, including \certifiedCount\ certified optima (\cref{sec:results}).
  \item In ablations on G81 and G72, no run without exchange or with random cluster exchange reached the optimum within \num{\abRunLimit} sweeps per run, \abBudgetRatio\ times the \num{\mainRunLimit}-sweep limit used in the main experiments.
  On G81, fixed labels, ranking by ascending cut, quenched states and discarding the complement made the search slower (\cref{sec:ablations}).
  \item An integer implementation runs on a \SI{\mcuClock}{\mega\hertz} microcontroller.
  A timing model, checked against eight runs on the board, gives a median of \mcuMedian\ minutes for G72 (\cref{sec:results-mcu}).
\end{itemize}

\section{Method}
\label{sec:method}

\subsection{Problem and single-temperature dynamics}
\label{sec:method-algorithm}

Let $G=(V,E)$ be an undirected graph without self-loops, with $N=|V|$ vertices, $M=|E|$ edges and integer weights $w_{ij}=w_{ji}$.
A state is a vector $x\in\{0,1\}^N$, and its cut is
\begin{equation}
  \mathrm{cut}(x)=\sum_{\{i,j\}\in E} w_{ij}\,[x_i\neq x_j],
  \label{eq:cut}
\end{equation}
where $[\cdot]$ is 1 if its condition holds and 0 otherwise.
With spins $\sigma_i=1-2x_i$, maximizing the cut is the same as minimizing the Ising energy $H(\sigma)=\sum_{\{i,j\}\in E} w_{ij}\sigma_i\sigma_j=W_\mathrm{tot}-2\,\mathrm{cut}(x)$, where $W_\mathrm{tot}$ is the sum of all weights~\cite{barahona1988application}.
The cut does not change under the global flip $x\mapsto\bar x=1-x$.
Flipping vertex $i$ changes the cut by an integer gain $g_i(x)$, which on a toroidal $\pm1$ lattice lies in $\{-4,-2,0,2,4\}$.
We measure the temperature $T>0$ in units of the cut, so the corresponding Ising temperature is $2T$.

A run holds a population of $R$ replicas, each with a uniformly random initial state.
The run maintains the best state found so far, called the incumbent, and alternates two kinds of work (\cref{fig:schematic}(a) and \cref{alg:strgcx}).

\emph{Metropolis sweeps.}
In one Metropolis sweep~\cite{metropolis1953equation}, a replica visits its $N$ vertices in a fresh random order and proposes each flip once.
A flip is accepted with probability $\min(1,e^{g_i/T})$, except that moves with $g_i<-20T$ are rejected outright.
All replicas use the same temperature throughout.
The cutoff changes each acceptance probability by at most $e^{-20}\approx\num{2e-9}$.

\emph{\rgcx\ calls.}
Let $s$ denote the number of completed sweeps per replica, $s_0$ the exchange-start threshold, and $k$ the exchange interval.
\rgcx\ calls are scheduled every $k$ sweeps after the exchange-start threshold $s_0$.
The call ranks the replicas by their current cut and applies the pair operation to every pair of ranks in a fixed order (\cref{sec:method-properties}).
The incumbent is compared with a replica after every accepted move and with every replica after every exchange call.
A run stops upon reaching the target cut or after a maximum of $L$ sweeps per replica.

\begin{figure}[!htb]
\centering
\includegraphics[width=\textwidth]{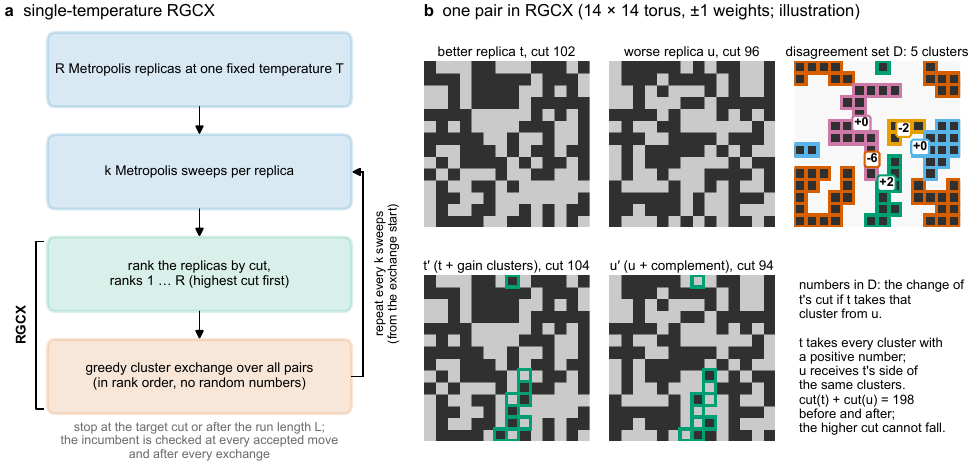}
\caption{\Strgcx. (a)~One run: Metropolis sweeps at one temperature and, every $k$ sweeps, an exchange call that ranks the replicas by cut and processes all pairs. (b)~One pair on a $14\times14$ torus with $\pm1$ weights, as an illustration. The disagreement set $D$ splits into five clusters, each marked with the change of the better replica's cut if it takes that cluster from the worse replica. The better replica takes every cluster with a positive change, the worse replica receives the complement, and the summed cut stays 198.}
\label{fig:schematic}
\end{figure}

\begin{algorithm}[!htb]
\caption{\Strgcx: one run.}
\label{alg:strgcx}
\begin{algorithmic}[1]
\Require graph $G$, number of replicas $R$, temperature $T$, interval $k$, start $s_0$, length $L$, target $c_\mathrm{t}$
\State draw $R$ replicas uniformly at random, and let the incumbent $x^\ast$ be the best of them
\For{$s=1,\dots,L$}
  \State give every replica one Metropolis sweep at temperature $T$, and update $x^\ast$
  \If{an \rgcx\ call is scheduled}
    \State rank the replicas by cut, highest first
    \For{each pair of ranks $(p,q)$, $p<q$, in lexicographic order}
      \State $\textsc{PairExchange}$ on the replicas ranked $p$ and $q$
    \EndFor
    \State update $x^\ast$
  \EndIf
  \State \textbf{if} $\mathrm{cut}(x^\ast)\geq c_\mathrm{t}$ \textbf{then} stop
\EndFor
\State \Return $x^\ast$
\Statex
\Procedure{PairExchange}{$a,b$}
  \State globally flip $b$ if it differs from $a$ on more than $N/2$ vertices
  \State $t\gets$ the replica with the higher cut ($a$ on a tie), $u\gets$ the other
  \For{each cluster $C$ of the set where $a$ and $b$ differ}
    \State \textbf{if} $t$ gains by taking the values of $u$ on $C$ \textbf{then} exchange $C$ between $t$ and $u$
  \EndFor
\EndProcedure
\end{algorithmic}
\end{algorithm}

\subsection{Ranked greedy cluster exchange}
\label{sec:method-properties}

At the start of each \rgcx\ call, the replicas are ranked by their current cuts, highest first, and the ranks stay fixed during the call.
The pairs are visited in lexicographic order of their ranks, $(1,2)$, $(1,3)$, \dots, $(1,R)$, $(2,3)$, \dots, $(R-1,R)$, so the replica with the highest cut meets every other replica first.
Ties in the cut ranking are broken by replica index.
Ranks determine the processing order, whereas the better replica in each pair is identified from the current cuts.

Consider a pair $(a,b)$, already aligned so that their disagreement set $D$ holds at most $N/2$ vertices ($b$ is unchanged at equality).
Let $C_1,\dots,C_m$ be the clusters of $D$, the connected components of the subgraph induced by $D$.
The two replicas agree outside $D$, no edge joins two different clusters, and an edge inside a cluster is cut in $a$ exactly when it is cut in $b$.
The cut of either replica therefore splits into a common part $A$ and one term per cluster,
\begin{equation}
  \mathrm{cut}(a)=A+\sum_{\ell=1}^m c_a(C_\ell),\qquad
  \mathrm{cut}(b)=A+\sum_{\ell=1}^m c_b(C_\ell),
  \label{eq:split}
\end{equation}
where $c_a(C_\ell)$ is the sum of the weights of the cut edges with at least one endpoint in $C_\ell$.
Let $t$ be the replica with the higher current cut ($a$ on a tie) and $u$ the other.
The gain of $t$ from taking the values of $u$ on $C_\ell$ is $\Delta_\ell=c_u(C_\ell)-c_t(C_\ell)$, a sum over the edges between $C_\ell$ and the rest of the graph.
The pair operation exchanges every cluster with $\Delta_\ell>0$ between the two replicas and leaves the others alone.

\subsection{Properties and computational cost}
\label{sec:method-order}

\begin{lemma}
\label{lem:pair}
Let $t'$ and $u'$ be the two replicas after the pair operation, the children of the parents $a$ and $b$.
\begin{enumerate}
  \item The summed cut is conserved: $\mathrm{cut}(t')+\mathrm{cut}(u')=\mathrm{cut}(a)+\mathrm{cut}(b)$.
  \item $\mathrm{cut}(t')=A+\sum_\ell\max\bigl(c_a(C_\ell),c_b(C_\ell)\bigr)\geq\max\bigl(\mathrm{cut}(a),\mathrm{cut}(b)\bigr)$.
  \item Among the $2^m$ states that take each cluster from one of the two parents, none has a higher cut than $t'$ or a lower cut than $u'$, its cluster-wise complement.
\end{enumerate}
\end{lemma}

\begin{proof}[Proof sketch]
\renewcommand{\qedsymbol}{}
Exchanging $C_\ell$ moves the amount $\Delta_\ell$ from one replica to the other, which gives claim (1).
Both replicas flip the same vertices, so their disagreement set stays $D$.
By \cref{eq:split}, a child's cut is $A$ plus one term per cluster taken from either parent, and the terms are independent.
The child that takes the larger term everywhere is $t'$, and the child that takes the smaller term everywhere is $u'$, which gives claims (2) and (3).
\end{proof}

Applied again to the same pair of children, without intervening changes to either replica, the operation exchanges nothing: their disagreement clusters are unchanged and none has a positive gain for the better child.

\rgcx\ is a deterministic optimization move that favors the better replica of each pair.
A pair operation that moves an amount $\delta>0$ from $u$ to $t$ keeps the summed cut and raises the sum of the squared cuts by $2\delta\,\bigl(\mathrm{cut}(t)-\mathrm{cut}(u)\bigr)+2\delta^2>0$.
An exchange call therefore keeps the mean cut of the population and raises the variance of its cuts whenever a cluster moves.
Whenever a strictly improving exchange has positive probability, a call increases the expected variance of the population's cuts and therefore cannot preserve the product Boltzmann distribution at $T$.
The best cut in the population never decreases within an exchange call.

Pair operations on disjoint replicas commute.
An \rgcx\ call costs $O(R^2(N+M))$, and at our paper's setting the exchange calls take \SIrange{\exchangeShareLo}{\exchangeShareHi}{\percent} of the wall time of a run on the desktop computer.

The exchange uses no random numbers, and the settings and random-number consumption do not depend on $L$.
For a fixed build and seed, the first $h$ sweeps of a longer run therefore match a run of length $h$, which lets us evaluate shorter cutoffs from one long run (\cref{sec:setup}).

\section{Experimental setup}
\label{sec:setup}

\paragraph{Instances.}
The main benchmark is the set of all 17 toroidal graphs of Gset~\cite{helmberg2000spectral}, with the graph files distributed at \url{https://web.stanford.edu/~yyye/yyye/Gset/}.
Our targets are their best known cuts.
For \certifiedCount\ of the 17 graphs, including G62, G72, G77 and G81, the upper bounds of Glass and Feder equal these values and prove them optimal~\cite{glass2026maximum}.
\Cref{sec:benchmark} lists all 17 graphs, their edge weights, and the corresponding upper bounds on the maximum cut.

\paragraph{Setting.}
All runs use $R=\Rrep$ replicas, $k=\kint$ sweeps between exchange calls and $T=\Tpaper$, unless an arm varies one of them.
The exchange start is $s_0=\startGeightyone$ for G81, $s_0=\startGseventytwo$ for G72 and G77, and $s_0=\startOther$ for the other Gset tori.
This setting was chosen in tuning runs, mainly on G81 with checks on G62, G72 and G77, with seeds not used in the experiments reported here.

\paragraph{Platforms and runs.}
The runs on the Gset tori under our default settings used an 8-core desktop computer (\desktopCPU, \SI{\desktopClock}{\giga\hertz} base frequency), with all eight cores working on one run at a time.
The ablations ran on a single node of the JUSUF high-performance computing (HPC) cluster~\cite{jsc2021jusuf}, equipped with two AMD EPYC 7742 processors (\clusterCores\ cores each, \SI{\clusterClock}{\giga\hertz} base frequency), one run per core.
Both arms of a comparison always ran on the same machine.
The arms of a comparison share the master seed and the seed indices, so a pair of runs shares its random initial conditions.
We used a C++ implementation on the desktop computer under Windows and the HPC cluster under Linux.
G81, G77, G72 and G62 have \num{2048} runs each, and the other 13 graphs have \EtwoRuns\ runs each.
Each run was limited to \num{\mainRunLimit} sweeps per replica.
The ablations have \abRuns\ paired runs per graph and arm, with each run limited to \num{\abRunLimit} sweeps per replica.

\paragraph{Cost measures.}
We report cost in wall seconds and replica-sweeps.
Wall seconds measure the latency of one run on its machine.
Replica-sweeps count the sweeps of all replicas and do not depend on the hardware, but they leave out the exchange, which takes most of the time (\cref{sec:method-order}).

\paragraph{Time to solution.}
Let $p_L$ be the fraction of runs that reach the target within $L$ sweeps, and $t_L$ the mean cost of a run of $L$ sweeps.
Our primary performance metric is the \ttsfixed,
\begin{equation}
  \mathrm{TTS99}(L)=t_L\,\max\!\left(1,\frac{\ln 0.01}{\ln(1-p_L)}\right),
  \label{eq:tts}
\end{equation}
which estimates the computational cost of independent runs of length $L$ needed to achieve a \SI{99}{\percent} probability of at least one success, with a minimum cost of one run.
If every run reaches the target, $p_L=1$ and the estimate is the cost of one run.
If none does, it is infinite, and we report that no run reached the target.
Every run is charged its full length $L$, including the sweeps after a hit.
Runs stop at the target, so for a run that hit before $L$ the cost of the remaining sweeps is computed from that run's own cost per sweep and per exchange call (\cref{sec:statistics}).
We choose $L^\ast$, the length that minimizes \cref{eq:tts}, on the runs with even seed index and report $\mathrm{TTS99}(L^\ast)$ on those with odd seed index, so that the choice does not bias the estimate.

For paired ablations, we use a restart-based TTS99: the 99th percentile of the total time to the first success when unsuccessful runs are restarted at a specified cutoff.
The cutoff is selected on one seed subset and evaluated on the other, with the roles reversed and the two estimates averaged.
We refer to this quantity as \ttsrenewal\ and evaluate it in both wall seconds and replica-sweeps.
\Cref{sec:statistics} gives the estimator details and the statistical decision rules.

\paragraph{Statistics.}
Unless otherwise stated, we report \SI{95}{\percent} bootstrap confidence intervals based on \num{2000} resamples of runs.
The confirmatory comparisons retain the original Holm correction (\cref{sec:statistics}).

\section{Results}
\label{sec:results}

\subsection{G81 and three other certified tori}
\label{sec:results-g81}

\Cref{tab:headline} gives the cost of G81 and three other certified tori on the desktop computer.
The estimated \ttsfixed\ for G81 is \SI[round-mode=places,round-precision=1]{\GeightyoneTTS}{\second} [\num[round-mode=places,round-precision=1]{\GeightyoneTTSlo}, \num[round-mode=places,round-precision=1]{\GeightyoneTTShi}], with a success rate of \GeightyonePodd\ at $L^\ast=\GeightyoneLstar$ sweeps.
The median first-hit time is \SI{\GeightyoneMedianHit}{\second}, and \SI{90}{\percent} of runs reach the optimum within \SI{\GeightyoneQninety}{\second} (\cref{fig:first-hit}(b)).
For G77, G72 and G62 the \ttsfixed\ lies between \SI{\GseventysevenTTS}{\second} and \SI{\GseventytwoTTS}{\second}.
A second set of 512 G81 runs with fresh seeds gave a \ttsfixed\ of \SI[round-mode=places,round-precision=1]{\GeightyoneTTSfresh}{\second} [\num[round-mode=places,round-precision=1]{\GeightyoneTTSfreshlo}, \num[round-mode=places,round-precision=1]{\GeightyoneTTSfreshhi}] on the 256 odd-indexed runs, with $L^\ast$ selected on the even-indexed runs.
The independent-seed estimates are close.
Recorded move events place every first hit in the runs of \cref{tab:headline} in an exchange call, never in a Metropolis move.

\begin{table}[!htb]
\centering
\caption{Time to solution of four certified Gset tori on the 8-core desktop computer. $L^\ast$ and $p$ (success rate within $L^\ast$) come from disjoint halves of the \num{2048} runs per graph (\cref{sec:setup}). Times are wall seconds. The \ttsfixed\ at $L^\ast$ carries a \SI{95}{\percent} interval. The last column is the median time of a single run to the optimum. \Cref{sec:statistics} gives the timing details.}
\label{tab:headline}
\begin{tabular}{@{}lrrrrr@{}}
\toprule
graph & $N$ & $L^\ast$ & $p$ & TTS99($L^\ast$) (s) & median first hit (s) \\
\midrule
G81 & \num{20000} & \GeightyoneLstar & \GeightyonePodd & \num[round-mode=places,round-precision=1]{\GeightyoneTTS}\ [\num[round-mode=places,round-precision=1]{\GeightyoneTTSlo}, \num[round-mode=places,round-precision=1]{\GeightyoneTTShi}] & \GeightyoneMedianHit \\
G77 & \num{14000} & \GseventysevenLstar & \GseventysevenPodd & \GseventysevenTTS\ [\GseventysevenTTSlo, \GseventysevenTTShi] & \GseventysevenMedianHit \\
G72 & \num{10000} & \GseventytwoLstar & \GseventytwoPodd & \GseventytwoTTS\ [\GseventytwoTTSlo, \GseventytwoTTShi] & \GseventytwoMedianHit \\
G62 & \num{7000} & \GsixtytwoLstar & \GsixtytwoPodd & \GsixtytwoTTS\ [\GsixtytwoTTSlo, \GsixtytwoTTShi] & \GsixtytwoMedianHit \\
\bottomrule
\end{tabular}
\end{table}

\begin{figure}[!htb]
\centering
\includegraphics[width=\textwidth]{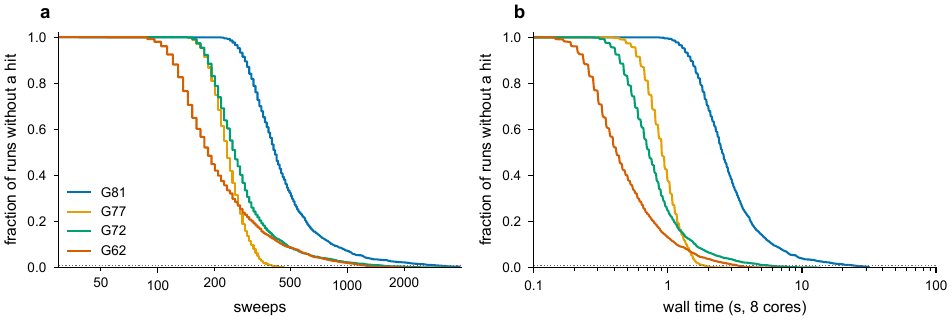}
\caption{First hits on the four tori of \cref{tab:headline}, \num{2048} runs each. (a)~Fraction of runs that have not reached the optimum, against sweeps. (b)~The same against wall time on the 8-core desktop computer. The dotted line in (a) and (b) marks \SI{1}{\percent} of runs without a hit.}
\label{fig:first-hit}
\end{figure}

\subsection{The toroidal benchmark and published comparisons}
\label{sec:results-class}
\label{sec:results-comparison}

On the other 13 toroidal graphs of Gset, all \EtwoRuns\ runs per graph reached the best known cut within 500 sweeps, with a \ttsfixed\ of \SIrange{\EtwoSmallLo}{\EtwoLargeHi}{\second} (\cref{sec:benchmark}).

\Cref{tab:published} lists the published times to the optimum of the certified tori that we found, each as its source defines it.
Among the published metrics considered here, Cosm's TTS is the most directly comparable to our \ttsfixed\ estimate~\cite{zick2026cosm}.
Both estimates use success probabilities and per-run costs rather than a directly measured first-hit latency.
Cosm v2 reports TTS values calculated from success probabilities and average per-trial times under an assumed configuration of 32 CPU cores with two threads per core.
Our per-run cost is the wall time of one run using all 8 cores of the desktop computer.
Cosm's reported TTS is over \cosmRatioLowerGeightyone\ times that of \strgcx\ on G81, about \cosmRatioGseventyseven\ times on G77, and \cosmRatioGseventytwo\ times on G72.
These are ratios of reported values from different hardware and timing protocols.
The heuristic of Glass and Feder~\cite{glass2026maximum} reaches the optimum of all four graphs, but its times come without a success rate.
Among the sources we found, no method reports a shorter time to the optimum of any of the four graphs, whether or not its time comes with a success rate.
For G62, Glass and Feder provide the only published time we found.

\begin{table}[!htb]
\centering
\caption{Published times to the optimum of the certified tori, in seconds, each as its source defines it. \Strgcx: \ttsfixed\ at $L^\ast$, 8-core desktop computer. Cosm v2: TTS calculated using average per-trial times under an assumed 32-core, two-thread-per-core configuration (later versions of the preprint omit these times). Zick 2025: time to target from the mean trial time and the success rate, laptop with 6 MATLAB workers. Glass \& Feder: heuristic time, without a success rate. Their table specifies an Intel i5 at \SI{2.6}{\giga\hertz}. The hardware used for the G81 heuristic timing is not unambiguously specified. A dash means no value was found.}
\label{tab:published}
\begin{tabular}{@{}ccccc@{}}
\toprule
graph & \begin{tabular}[c]{@{}c@{}}Single-temperature\\\rgcx\end{tabular} & Cosm v2~\cite{zick2026cosm} & Zick 2025~\cite{zick2025performance} & Glass \& Feder~\cite{glass2026maximum} \\
\midrule
G81 & \num[round-mode=places,round-precision=1]{\GeightyoneTTS} & \num{\cosmGeightyone} & \num{\zickGeightyone} & \num{\gfGeightyone} \\
G77 & \GseventysevenTTS & \cosmGseventyseven & \num{\zickGseventyseven} & \gfGseventyseven \\
G72 & \GseventytwoTTS & \cosmGseventytwo & \num{\zickGseventytwo} & \gfGseventytwo \\
G62 & \GsixtytwoTTS & -- & -- & \gfGsixtytwo \\
\bottomrule
\end{tabular}
\end{table}

\subsection{Ablation of the main components}
\label{sec:ablations}

We compare \strgcx\ with six ablation variants, each modifying one component of the reference configuration defined in \cref{sec:setup}, on the same seed indices.
We performed \abRuns\ paired runs per arm on each of G81 and G72.
\Cref{fig:ablation} gives the ratios of the \ttsrenewal, arm over our paper's setting, and the outcome of the decision rule of \cref{sec:statistics}.
The random rule exchanges each cluster independently with probability one half.
Fixed labels give replica $i$ the rank $i$ in every call, and ascending labels rank by increasing cut.
The quench replaces the Metropolis sweeps by sweeps at $T=0.05$, where the acceptance cutoff rejects every proposed move that lowers the cut.
Discarding the complement restores the worse replica of each pair to its state before the pair operation.

\begin{figure}[!htbp]
\centering
\includegraphics[width=0.95\textwidth]{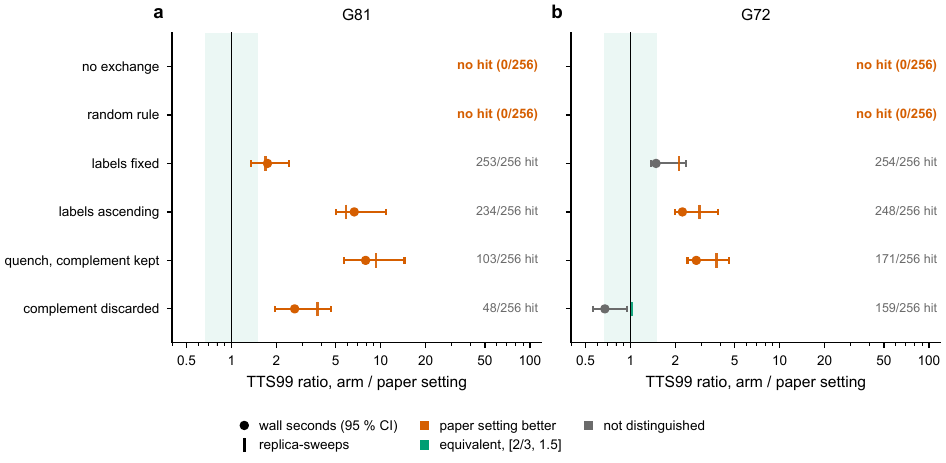}
\caption{Ablations on G81 (a) and G72 (b), \abRuns\ paired runs per arm. Each row specifies a modification of \strgcx. Dots: ratio of the \ttsrenewal\ in wall seconds, arm over our paper's setting, with \SI{95}{\percent} interval. Vertical ticks: the same ratio in replica-sweeps. Marker colors indicate the decision in the corresponding unit: circles for wall time and vertical ticks for replica-sweeps. The shaded band is the equivalence margin $[2/3,1.5]$. The number of the arm's runs that reached the optimum is given on the right, and arms in which none did within \num{10000} sweeps are marked ``no hit''. Equivalence is declared when the \SI{90}{\percent} interval of the ratio lies entirely within $[2/3,1.5]$ (\cref{sec:statistics}). On G81, all displayed comparisons except ascending labels belong to the original confirmatory family. Ascending labels and all G72 direction checks are exploratory.}
\label{fig:ablation}
\end{figure}

Without exchange calls, no run reached the optimum of G81 or G72 within \num{\abRunLimit} sweeps per run, \abBudgetRatio\ times the \num{\mainRunLimit}-sweep limit used in the main experiments.
The same holds for the random rule, which differs from Houdayer's original move of a single randomly chosen cluster.

Fixed labels made G81 slower by a factor of \fixedRS\ [\fixedRSlo, \fixedRShi] in replica-sweeps and \fixedWall\ [\fixedWalllo, \fixedWallhi] in wall seconds.
Ranking by ascending cut, so that the worst replica meets every other one first, made it slower by a factor of \ascRS\ in replica-sweeps and \ascWall\ in wall seconds.
On G72 both changes made the search slower in replica-sweeps, and ascending labels also in wall seconds.

Replacing the Metropolis sweeps by quenches made G81 slower by a factor of \qkeepRS\ in replica-sweeps and \qkeepWall\ in wall seconds, and only \qkeepHits\ of \abRuns\ runs reached the optimum.
Discarding the complement made G81 slower by a factor of \parentRS\ in replica-sweeps and \parentWall\ in wall seconds.
On G72 this arm was equivalent in replica-sweeps (\parentGseventytwoRS) and not distinguished in wall seconds (\parentGseventytwoWall), so a benefit of keeping the complement was shown on G81 only.

\subsection{A microcontroller}
\label{sec:results-mcu}

\begin{figure}[!htbp]
\centering
\includegraphics[width=\textwidth]{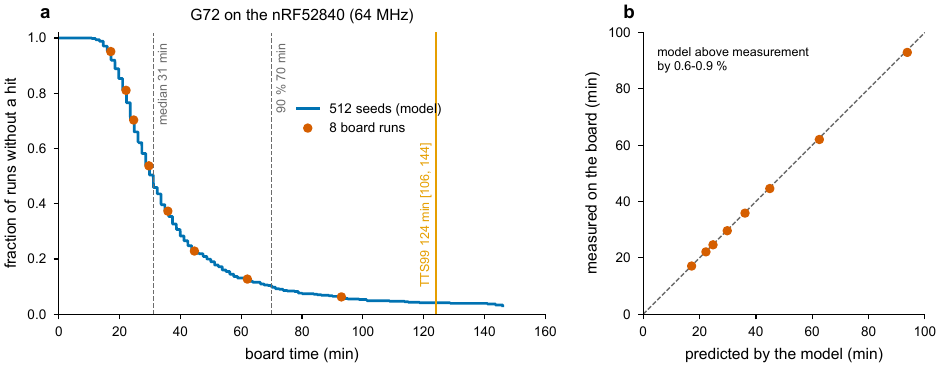}
\caption{G72 on the nRF52840 board at \SI{\mcuClock}{\mega\hertz}. (a)~Fraction of runs without a hit versus board time, using modeled times from 512 host runs (line) and measurements from 8 board runs (dots), with the median, the 90th percentile and the \ttsfixed. The model curve is truncated at the time corresponding to the \num{1000}-sweep run limit. (b)~Measured board time against the model's prediction for the 8 board runs. The dashed line is equality.}
\label{fig:mcu}
\end{figure}

We ran G72 at our paper's setting on a Seeed Studio XIAO nRF52840 Sense (\SI{\mcuClock}{\mega\hertz})
development board~\cite{xiao52840sense}, with a separate integer implementation that keeps all \Rrep\ replicas of the \num{10000}-vertex graph and their working data in a fixed arena of \SI{\mcuArena}{\kibi\byte} without heap memory.
The port uses a different random number generator and sweep order from the host implementation.

Because board runs are slow, we ran 512 seeds using a host-computer build of the microcontroller code.
A timing model calibrated on a separate board run converted their first-hit sweeps into board minutes.
Eight board runs, chosen to span the distribution, reproduced their host-run records bit for bit, and the model overestimated their times by \SIrange{\mcuModelLo}{\mcuModelHi}{\percent} (\cref{fig:mcu}(b)).
Each board run took between \mcuBoardLo\ and \mcuBoardHi\ minutes.

Of the 512 host runs, \mcuHits\ reached the optimum within \num{1000} sweeps.
By the model, a run reaches the optimum in a median of \mcuMedian\ minutes, and \SI{90}{\percent} of runs do so within \mcuPninety\ minutes (\cref{fig:mcu}(a)).
The \ttsfixed\ is \mcuTTS\ minutes [\mcuTTSlo, \mcuTTShi] at $L^\ast=\mcuLstar$ sweeps (\cref{sec:statistics}).

\begin{samepage}
These results show that the full population search can run on a general-purpose microcontroller with a small working memory.
\end{samepage}

\section{Discussion and conclusion}
\label{sec:discussion}

The results suggest that cluster exchange redistributes favorable clusters between replicas, while Metropolis updates continue to generate new configurations.

By \cref{lem:pair} the pair operation is a ratchet: the better replica of a pair never loses, and the summed cut is conserved, so the exchange selects rather than samples.
With the same clusters exchanged at random instead, no run reached the optimum of G81 or G72 within the ablation budget, supporting the importance of greedy cluster selection.
The Metropolis sweeps supply new clusters, the ranking decides which replicas meet first, and the population supplies the partners. The gain from \rgcx\ depends on the decomposition of the aligned disagreement set into clusters and on whether exchanging any of these clusters increases the better replica's cut.

\Strgcx\ returns candidate cuts without certifying their optimality.
A matching upper bound, such as those computed by Glass and Feder~\cite{glass2026maximum}, certifies a candidate's optimality.

\Strgcx\ combines fixed-temperature Metropolis sweeps with ranked greedy exchange, using cooperation between replicas to sustain the search at one temperature.
On an 8-core desktop computer the estimated \ttsfixed\ for G81 is \SI[round-mode=places,round-precision=1]{\GeightyoneTTS}{\second}, about two orders of magnitude below the shortest published TTS99.
With the same population size, exchange interval and temperature, the method reaches the best known cuts on all 17 toroidal Gset graphs, covering both $\pm1$ and unit edge weights.
The integer implementation also reaches the optimum of G72 on a \SI{\mcuClock}{\mega\hertz} microcontroller, demonstrating feasibility under tight memory constraints.
\Strgcx\ provides an effective population-based method for Max-Cut on two-dimensional toroidal lattices.

\begingroup
\small
\bibliographystyle{unsrtnat}
\bibliography{refs}
\endgroup

\clearpage
\appendix
\section*{Appendix}
\renewcommand{\thetable}{\thesection\arabic{table}}
\setcounter{table}{0}
\section{Statistical and timing details}
\label[appendix]{sec:statistics}
The at-least-one-run rule in \cref{eq:tts} also applies to every bootstrap draw and to the choice of $L^\ast$.
A direct check on 256 runs of G81 without the stop at the target gave a mean cost within \SI{\VoneDeviation}{\percent} of this estimate.
We select $L^\ast$ from the sweeps that end an exchange call, up to \num{2000} or the run limit, whichever is smaller.
For fixed-length TTS99, bootstrap intervals resample the odd-indexed evaluation runs while holding the selected $L^\ast$ fixed.

\paragraph{Renewal TTS99.}
The renewal TTS99 models a solver that restarts a run whenever it reaches a cutoff $L_\mathrm{c}$ without a hit.
Let $t_1\leq\dots\leq t_n$ be the first-hit times of the $n$ runs, with $t_i=\infty$ for a run without a hit, and let $q$ be the fraction of runs with $t_i\leq L_\mathrm{c}$.
Let $j$ be the largest nonnegative integer with $(1-q)^j>0.01$, and $r$ the empirical quantile of the $t_i$ at level $1-0.01/(1-q)^j$.
The renewal TTS99 is then $jL_\mathrm{c}+r$, the \SI{99}{\percent} quantile of the total time, with the last partial attempt included.
If every run hits within $L_\mathrm{c}$, it is the \SI{99}{\percent} quantile of the $t_i$.
If no run hits, the estimate is infinite, and the arm is reported as without hits.
Times are counted in replica-sweeps, with cutoffs on a grid of 10 sweeps.
Cutoffs are chosen on the runs with even seed index and evaluated on those with odd index, and the other way round, and the two estimates are averaged.
Paired ratios in wall seconds apply the estimator directly to first-hit times in seconds, with cutoffs on a grid in seconds, rather than dividing separately converted single-arm estimates.

\paragraph{Paired comparisons.}
Two arms are compared by the ratio of their \ttsrenewal, with a paired bootstrap over seed indices, in replica-sweeps and in wall seconds.
Each bootstrap draw resamples paired runs within each seed-parity half and repeats the cutoff selection and cross-evaluation.
In each unit, arm B is slower than arm A if the ratio B/A is at least 1.5, its \SI{95}{\percent} interval lies above 1, and a sign test on which run of a pair hits first, counted in sweeps, gives an unadjusted $p<0.01$.
Arm B is faster under the mirrored conditions: a ratio of at most $2/3$, an interval below 1 and the same sign test.
The arms are equivalent in a unit if the \SI{90}{\percent} interval of the ratio lies within $[2/3,1.5]$, and any other outcome is reported as not distinguished.
Unqualified statements of slowing or equivalence require agreement in both units.
Otherwise, the unit is specified.

For each experiment, the design, the confirmatory comparisons and these decision rules were recorded before it ran.
In the ablations, the $p$-values of the confirmatory comparisons are also adjusted with Holm's method~\cite{holm1979simple} at a family-wise level of 0.05.
Every confirmatory outcome reported in \cref{sec:ablations} holds after this adjustment, and intervals are reported unadjusted.
Comparisons not named as confirmatory are exploratory.

The Holm correction retains the original six G81 comparisons: no exchange, the random rule, fixed labels, discarding the complement, quenched states and one soft-exchange arm not displayed here.
G72 comparisons and ascending labels were exploratory.
Arms with no successful runs are compared with the reference configuration using an exact McNemar test on paired success outcomes.
Each arm of a comparison ran as one job, so the intervals of wall-second ratios resample runs within those jobs and do not include the variation of the cost per sweep between separate jobs.

\paragraph{Microcontroller timing.}
The time model for \cref{sec:results-mcu} gives a run's board time as
\begin{equation*}
  t_\mathrm{board}=\SI{2.757}{\second}\times s+\SI{2.502}{\micro\second}\times v,
\end{equation*}
where $s$ is the number of sweeps of the population and $v$ the number of vertices visited in the exchange calls.
It was calibrated on the last G72 board run of the port's development, separate from the eight validation runs.
Before a recorded run ends, visits at $L$ are scaled by the fraction of its total host exchange time elapsed; beyond its end, remaining calls use the mean visits per call over its last five calls.
The eight board runs in \cref{fig:mcu} used seeds chosen at quantiles 0.05 to 0.97 of the hit-sweep distribution among the successful host runs.

\section{Complete toroidal benchmark results}
\label[appendix]{sec:benchmark}
The 17 toroidal graphs of Gset are G11--G13 ($N=800$), G32--G34 ($N=\num{2000}$), G48--G50 ($N=\num{3000}$), G57 ($N=\num{5000}$), G62 ($N=\num{7000}$), G65--G67 ($N=\num{8000}$--\num{10000}), G72 ($N=\num{10000}$), G77 ($N=\num{14000}$) and G81 ($N=\num{20000}$).
All are two-dimensional periodic lattices of degree four.
G48--G50 have unit weights, and G48 and G49 are bipartite, so their optimum is the number of edges.
The other 14 graphs have weights $\pm1$.
For \certifiedCount\ graphs (G11, G48, G49, G57, G62, G65, G72, G77 and G81), the upper bound of Glass and Feder~\cite{glass2026maximum} equals the best known cut.
For the other eight, the bounds exceed the best known cut by \boundGapLo\ to \boundGapHi\ for seven graphs and by \boundGapGfifty\ for G50.
For comparison across graphs we also give the \ttsfixed\ at $L=500$ on all runs.

\begin{table}[!htbp]
\centering
\caption{All 17 Gset tori on the 8-core desktop computer. C: certified optimum; B: best known cut. Times are wall seconds. $L^\ast$ is chosen on even seed indices and evaluated on odd ones, with bootstrap \SI{95}{\percent} intervals. TTS99(500) uses all runs. G62, G72, G77 and G81 have \num{2048} runs each, and the other graphs \EtwoRuns\ each.}
\label{tab:benchmark}
\small
\setlength{\tabcolsep}{4pt}
\begin{tabular}{@{}lrrcrrr@{}}
\toprule
graph & $N$ & target & status & $L^\ast$ & TTS99($L^\ast$) (s) & TTS99(500) (s) \\
\midrule
G11 & 800 & 564 & C & 56 & 0.00728 [0.00725, 0.00732] & 0.191 \\
G12 & 800 & 556 & B & 56 & 0.00865 [0.00705, 0.0103] & 0.185 \\
G13 & 800 & 582 & B & 64 & 0.0104 [0.0103, 0.0108] & 0.184 \\
G32 & \num{2000} & \num{1410} & B & 72 & 0.0298 [0.0297, 0.03] & 0.409 \\
G33 & \num{2000} & \num{1382} & B & 80 & 0.0376 [0.0374, 0.039] & 0.416 \\
G34 & \num{2000} & \num{1384} & B & 72 & 0.0317 [0.0302, 0.039] & 0.416 \\
G48 & \num{3000} & \num{6000} & C & 56 & 0.0153 [0.015, 0.0157] & 0.253 \\
G49 & \num{3000} & \num{6000} & C & 56 & 0.0153 [0.015, 0.0157] & 0.240 \\
G50 & \num{3000} & \num{5880} & B & 88 & 0.0548 [0.0527, 0.0676] & 0.489 \\
G57 & \num{5000} & \num{3494} & C & 144 & 0.488 [0.418, 0.57] & 0.976 \\
G62 & \num{7000} & \num{4872} & C & \GsixtytwoLstar & \GsixtytwoTTS\ [\GsixtytwoTTSlo, \GsixtytwoTTShi] & 2.51 \\
G65 & \num{8000} & \num{5562} & C & 160 & 0.537 [0.418, 0.639] & 1.53 \\
G66 & \num{9000} & \num{6364} & B & 168 & 0.485 [0.483, 0.504] & 1.72 \\
G67 & \num{10000} & \num{6950} & B & 232 & 0.831 [0.8, 1.03] & 1.90 \\
G72 & \num{10000} & \num{7008} & C & \GseventytwoLstar & \GseventytwoTTS\ [\GseventytwoTTSlo, \GseventytwoTTShi] & 3.25 \\
G77 & \num{14000} & \num{9940} & C & \GseventysevenLstar & \GseventysevenTTS\ [\GseventysevenTTSlo, \GseventysevenTTShi] & 2.45 \\
G81 & \num{20000} & \num{\GeightyoneOpt} & C & \GeightyoneLstar & \num[round-mode=places,round-precision=1]{\GeightyoneTTS} [\num[round-mode=places,round-precision=1]{\GeightyoneTTSlo}, \num[round-mode=places,round-precision=1]{\GeightyoneTTShi}] & 13.3 \\
\bottomrule
\end{tabular}
\end{table}
For the eight graphs marked B, the upper bounds of Ref.~\cite{glass2026maximum} are 558 (G12), 583 (G13), \num{1414} (G32), \num{1384} (G33), \num{1386} (G34), \num{6000} (G50), \num{6366} (G66) and \num{6952} (G67).

\end{document}